\documentclass[11pt]{amsart}
\usepackage{amsmath}
\usepackage{amsfonts}
\usepackage{amssymb}
\usepackage{amscd}
\usepackage[arrow,matrix,curve,cmtip,ps]{xy}
\usepackage{graphicx}

\usepackage{framed,color}
\definecolor{shadecolor}{rgb}{1,0.8,0.3}

\numberwithin{equation}{section}

\usepackage{amsthm}

\theoremstyle{plain}
\newtheorem{thm}{Theorem}[section] 
\newtheorem{theorem}[thm]{Theorem}

\theoremstyle{definition}

\theoremstyle{remark}

\newcommand{\cl}[1]{\mathcal{#1}}

\newcommand{\bb}[1]{\mathbb{#1}}

\begin{document}

\title{Relative Weak Injectivity for Operator Systems}
\author{Ali S$.$ Kavruk}
\thanks{2010 Mathematics Subject Classification. Primary 46L06, 46L07; Secondary 46L05, 47L25, 47L90.}
\thanks{Key words. operator system, tensor product, nuclear C*-algebras, injectivity}
\begin{abstract}

We investigate the notion of relative weak injectivity and its  nuclearity related properties in the category of operator systems.  We obtain several characterizations of the weak expectation property. We show that (c,max)-nuclearity characterizes Kirchberg and Wasserman's C*-systems. Namioka and Phelps' test systems, which detects nuclear C*-algebras, is shown to characterize nuclear C*-systems.
We study quasi-nuclearity in the operator system setting and prove that quasi-nuclearity and nuclearity are equivalent, in other words, (er,max)-nuclearity and (min,max)-nuclearity are equivalent.

\end{abstract}

\maketitle

Weak expectation property (WEP) is a fundamental nuclearity related property introduced by C$.$ Lance \cite{Lance2}, \cite{Lance}. For von Neumann algebras WEP coincides with injectivity, and can be interpreted as a weak form of classical completely positive factorization property. If a C*-algebra $\cl A$ has WEP then $\cl A^{**}$ is injective relative to $\cl A$, therefore such objects are also known as weakly injective C*-algebras \cite{Kirchberg-Presentation}, \cite{Paulsen WEP}. WEP is studied in various categories in operator algebras \cite{pisier_intr}, \cite{Haagerup},  \cite{Blecher},  \cite{KPTT Nuclearity},  \cite{JIAN LIANG}. A more general framework of weak injectivity introduced by E$.$ Kirchberg \cite{Kirchberg94}:  for C*-algebras $\cl A \subseteq \cl B$, $\cl A$ is said to be \textit{relatively weakly injective} in $\cl B$, or weakly relatively injective, $w.r.i.$ in short, if the canonical inclusion of $\cl A$ into $\cl A^{**}$ admits a conditional expectation on $\cl B$. This is a nuclearity related property in the sense that $\cl A$ is w.r.i$.$ in $\cl B$ if and only if $\cl A \otimes_{\max} \cl C \subseteq \cl B \otimes_{\max} \cl C$ for every C*-algebra $\cl C$, in other words, the projective tensor product behaves injectively. A recent work \cite{Kavruk NP} exposes that relative weak injectivity coincides with tight Riesz interpolation property and Riesz-Arveson extension property for C*-algebras.

\smallskip

Our main purpose in this paper is to study the notion of relative weak injectivity in the operator system category. As in weak expectation property, weak relative injectivity has two natural extensions to operator systems, namely w.r.i. and r.d.c.i., \textit{relative double commutant injectivity}. We define w.r.i$.$ for a pair of operator system $\cl S_1 \subseteq \cl S_2$ analogously: $\cl S_1$ is said to be \textit{w.r.i}$.$ in $\cl S_2$ if the canonical inclusion $\cl S_1$ into $\cl S_1^{**}$ extends to a ucp map on $\cl S_2$. A weaker condition is defined as follows: $\cl S_1$ has \textit{r.d.c.i}$.$ in $\cl S_2$ if for all representation $\cl S_1 \subseteq B(\cl H)$, the inclusion $i$ of $\cl S_1$ into $B(\cl H)$ extends to a ucp map $\tilde i$ on $\cl S_2$ such a way that $\tilde i(\cl S_2) \subseteq i(\cl S_1)''$, the bi-commutant of the image of $i$. It is worth mentioning that in the definition of r.d.c.i., embedding of $\cl S_1$ into {\it every} $B(\cl H)$ is essential, in fact, one can always find an inclusion $\cl S_1 \subseteq B(\cl H)$ such that the requirement in the definition holds. W.r.i$.$ and r.d.c.i$.$ are different in operator system category, however, coincide for a pair of C*-algebras.

\smallskip

\textbf{Note:} The notion of relative weak injectivity is also studied previously in \cite{AB}. We remark that the w.r.i$.$ definition proposed in \cite{AB} coincides with r.d.c.i$.$ definition in this article.

\smallskip

We first observe that w.r.i$.$ is a nuclearity related property: for $\cl S_1 \subseteq \cl S_2$, $\cl S_1$ is w.r.i$.$ in $\cl S_2$ if and only if $\cl S_1 \otimes_{\max} \cl T \subseteq \cl S_2 \otimes_{\max} \cl T$ for every operator system $\cl T$. One of our characterization for w.r.i$.$ involves in extension property into matrix systems: $\cl S_1 $ is w.r.i.\! in $\cl S_2$ if and only if for every matrix system $\cl R$ (i.e$.$ an operator subsystem of a matrix algebra $M_n$ for some $n$) every ucp map $\varphi: \cl S_1 \rightarrow \cl R$, has a ucp extension $\tilde \varphi: \cl S_2 \rightarrow \cl R$. This allows us to obtain a tensorial characterization via matrix quotients: $\cl S_1$ is w.r.i$.$ in $\cl S_2$ if and only if $\cl S_1 \otimes_{\max} (M_n / J) \subseteq \cl S_2 \otimes_{\max} (M_n / J)$ for every $n$ and null-subspace $J\subset M_n$. As a corollary we obtain that $\cl S$ has the weak expectation property if and only if the minimal and the maximal tensor product coincide on $\cl S \otimes (M_n / J)$ for every $n$ and null-subspace $J \subset M_n$. As the quotient has the lifting property, we recover the equivalence of (el,max)-nuclearity and weak expectation property.


\smallskip

Arveson's extension theorem is one of the most fundamental tool in operator algebras \cite{Arveson}. Combining with a result of Choi and Effros, it states that for operator systems $\cl S_1 \subseteq \cl S_2$ and Hilbert space $\cl H$, every cp map $\varphi: \cl S_1 \rightarrow B(\cl H)$ extends to cp map $\tilde \varphi: \cl S_2 \rightarrow B(\cl H)$. There is also a partial order on the self-adjoint completely bounded maps given by $\psi \leq \phi$ if $\phi-\psi$ is a cp map. Relative weak injectivity is the key property for Arveson's extension theory which also obeys the order of cp maps: $\cl S_1 \subseteq \cl S_2$ is w.r.i$.$ if and only if every state $\varphi$ on $\cl S_1$ has a state extension $\tilde \varphi$ on $\cl S_2$ such that every positive linear functional $\psi$ on $\cl S_1$ with $\psi \leq \varphi$ has a positive extension $\tilde \psi$ on $\cl S_2$ with $\tilde \psi \leq \tilde \varphi$. Moreover, this can be achieved such a way  the  $E: [\varphi] \rightarrow [\tilde \varphi]$ given by $\psi \mapsto \tilde \psi$ is a ucp map for which the restriction map is the inverse ucp map. (Here $[\varphi]$ denotes the Effros system associated with $\varphi$ as explained in Section 1.)

\smallskip

An operator system $\cl S$ for which the bidual operator system $S^{**}$ has structure of a C*-algebra is called a {\it C*-system} \cite{KW}. Nuclear operator systems are such examples, in fact, if $\cl S$ is nuclear then the bidual operator system $\cl S^{**}$ is injective, thus has a structure of an injective von Neumann algebra \cite{Kirchberg94}. We prove that being a C*-system is a nuclearity related property and is equivalent to (c,max)-nuclearity in the sense of \cite{KPTT Nuclearity}. In fact, $\cl S$ is a C*-system if and only if $\cl S$ is w.r.i$.$ in its universal C*-algebra $C_u^*(\cl S)$.

\smallskip

In tensor theory of compact convex sets the square is a test object to verify semi-simplexity \cite{NP}. Therefore, from the nuclearity viewpoint, Namioka and Phelps' test systems identify the nuclear objects in function systems. Recall that these test systems are defined by
$$
\cl W_{2n} = \{ (a_i) \in \ell^\infty_{2n} :\; a_1 + a_2 +  \cdots + a_n = a_{n+1} + a_{n+2} + \cdots a_{2n} \} \subseteq  \ell^\infty_{2n}.
$$
In \cite{kav wri C} we streamline their results in non-commutative setting by showing that $\cl W_6$ detects  nuclear C*-algebras. We extend this result to C*-systems: a C*-system $\cl S$ is nuclear if and only if we have a canonical complete order isomorphism
$
\cl S\otimes_{\min} \cl W_6 = \cl S\otimes_{\max} \cl W_6.
$
Unfortunately such a property for general operator systems remains open. However we prove that an operator system $\cl S$ is nuclear if and only if $
\cl S \otimes_{\min} \cl R = \cl S \otimes_{\max} \cl R
$ for every matrix system $\cl R$.
This, in particular, implies that (er,max)-nuclearity is equivalent to nuclearity, extending C. Lance's notion of quasi-nuclearity to general operator systems. 

\smallskip

In Section 1 we review basics aspects of operator systems that are required for our work herein. In Section 2 we obtain several equivalent formulations of relative weak injectivity and deduce that C*-systems and (c,max)-nuclear objects coincide. Section 3 includes non-commutative analogue of Namioka and Phelps' theory extended on C*-system and quasi-nuclearity for operator systems.

\section{Preliminaries}

By an \textit{operator system} $\cl S$ we mean a unital $*$-closed subspace of $ B(\cl H)$ together with the induced matricial order structure, where $B(\cl H)$ denotes the von Neumann algebra of bounded linear operators on a Hilbert space $\cl H$.  We refer the reader to \cite{PaulsenBook} for an excellent source of operator systems and their abstract characterizations given by Choi and Effros \cite{ChoiEffros77}. A map between operator systems
$ \varphi: \cl S_1 \rightarrow \cl S_2$ is called \textit{completely positive}, \textit{cp} in short, if the $n^{\rm th}$-amplification
$id_n \otimes \varphi: M_n \otimes \cl S_1 \rightarrow M_n \otimes \cl S_2$ is positive for all $n$.
If $\varphi$ is also unital, i.e. $\varphi(e) = e$, we will say that $\varphi$ is a \textit{ucp} map. We assume familiarity with the universal and enveloping C*-algebra of an operator system $\cl S$, denoted by $C_u^*(\cl S)$ and $C_e^*(\cl S)$, resp. For a cp map  $\varphi: \cl S\rightarrow B(\cl H)$
$$
[\varphi] = span\{\psi: \cl S \rightarrow B(\cl H): \; \psi \mbox{ is cp with } \psi\leq \varphi\} \subseteq CB(\cl S, B(\cl H))
$$
is an operator system with unit $\varphi$, which we call the \textit{Effros system} associated with $\varphi$.

\subsection{Duality} The topological dual $\cl S^*$ of an operator system $\cl S$ can be endowed with a matricial
order structure via matricial cp maps. More precisely, after defining the self-adjoint idempotent $*$ via $f^*(s) = \overline{f(s^*)}$, we declare
$$
(f_{ij}) \in M_n(\cl S^*) \mbox{ positive if } \cl S \ni s \longmapsto (f_{ij}(s)) \in M_n \mbox{ is cp}. 
$$
The collection of the cones of the positive elements $\{M_n(\cl S^*)^+\}_{n=1}^\infty$ forms a strict, compatible matricial order structure on $\cl S^*$.
In general an Archimedean matrix order may fail to exist for this matricially ordered space. If dim$(\cl S) < \infty$ then a faithful state $w$ on $\cl S$ can be assigned as an Archimedean order unit for $\cl S^*$ \cite{ChoiEffros77}. In general, for any state $f$ on $\cl S$
$$
[f] = span \{g: g\leq f\} \subseteq S^*
$$
declares an operator system with unit $f$, which we call the Effros system as above.

\subsection{Quotients} A subspace $J \subset \cl S$ is called a kernel if $J$
is kernel of a ucp map defined from $\cl S$ (equivalently kernel of a cp map).
A kernel is typically a non-unital $*$-closed subspace but these properties, in general, do not characterize a kernel.
A matricial order structure on the algebraic quotient $\cl S/J$ can be defined
by
$$
Q_n = \{ (s_{ij} + J): (s_{ij}) \in M_n(\cl S)^+) \}.
$$
The Archimedeanization process, i.e, completion of the
cones $\{Q_n\}$ relative to order topology induced by $(e +J) \otimes I_n$ (see \cite{Paulsen-Tomforde}, \cite{KPTT Tensor}),
yields the operator system quotient $\cl S/J$.
The universal property of the quotient ensures that if $\varphi: \cl S \rightarrow \cl T$
is a ucp map then the induced map $\dot{\varphi}: \cl S/{\rm ker}(\varphi) \rightarrow \cl T$ is again a ucp map \cite{farenick--paulsen2011}.
$\varphi$ is called a \textit{quotient}  (resp.,\textit{ complete quotient}) map if $\dot{\varphi}$ is an order (resp. a complete order)  inclusion.
In particular a surjective ucp map $\varphi$ is completely quotient if and only if the adjoint $\varphi^\dag: \cl T^* \rightarrow \cl S^*$ is a complete order inclusion.

A finite dimensional $*$-closed subspace $J$ of an operator system $\cl S$ which does not include any positive elements other that 0 is called a \textit{null-subspace}. Any null-subspace is a kernel \cite{kavruk2011}. If $\cl R \subseteq M_n$ is an operator system then $\cl R^* \cong M_n/J$ for some null-subspace $J\subset M_n$. To see this one can simply take the adjoint $i^\dag: M_n^* \rightarrow \cl R^*$ of the inclusion of $R$ into $M_n$, which is a quotient map. Now using Farenick and Paulsen's identity $M_n^* \cong M_n$ \cite{farenick--paulsen2011} and observing that the kernel of $i^\dag$ is a null-subspace we get $\cl R^* \cong M_n/J$. Conversely $(M_n/J)^*$ can be embedded into $M_n$ as an operator subsystem. To see this one can take the adjoint of the quotient map from $M_n$ into $M_n/J$. We leave the details to the reader.

\subsection{Minimal tensor product} For operator systems $\cl S$ and $\cl T$ we define
$$
C_n^{\min} = \{ [x_{ij}] \in M_n(\cl S \otimes \cl T): \; [(\phi \otimes \psi)(x_{ij})] \geq 0 \hspace{3cm}
$$
$$
\hspace{3cm}\mbox{ for all }p,q \in \bb N, \mbox{ for all ucp } \phi: \cl S \rightarrow M_p \mbox{ and } \psi: \cl T \rightarrow M_q\}.
$$
The collection of cones $\{C_n^{\min}\}_{n=1}^\infty$ forms a strict compatible matricial ordering for the algebraic tensor $\cl S \otimes \cl T$.  
Moreover, $1_{\cl S} \otimes 1_{\cl T}$ is a  Archimedean order unit. Therefore the triplet
$
(\cl S \otimes \cl T, \{C_n^{\min}\}_{n=1}^\infty, 1_{\cl S} \otimes 1_{\cl T})
$
forms an operator system which we call the minimal tensor product of $\cl S$ and $\cl T$ and denote by $\cl S\otimes_{\min} \cl T$. 
We refer the reader to \cite{KPTT Tensor} for details. The minimal tensor product is spatial, injective and functorial. By the representation
of the minimal tensor we mean the identification 
CP$(\cl S, \cl T) \cong (\cl S^*\otimes_{\min} \cl T)^+
$
whenever dim$(\cl S)<\infty$ \cite{kavruk2011}.

\subsection{Maximal tensor product} Let $\cl S$ and $\cl T$ be two operator systems. We define
$$
D_n^{\max} = \{ X^* (S \otimes T) X:  \; S \in M_p(\cl S)^+, \; T \in M_q(\cl T)^+,\; X \mbox{ is }pq\times n \mbox{ matrix}, \; p,q \in \bb N \}.
$$
The collection of the cones $\{D_n^{\max}\}_{n=1}^\infty$ are strict and compatible. Moreover, $1\otimes 1$ is a matricial order unit
for the matrix ordered space $(\cl S \otimes \cl T, \{D_n^{\max}\})$. Nonetheless $1\otimes 1$ may fail to be Archimedean, which can be resolved by
Archimedeanization process. We define
$$
C_n^{\max} = \{  X \in M_n(\cl S \otimes \cl T): X + \epsilon (1\otimes 1)_n \in D_n^{\max} \mbox{ for all } \epsilon >0    \}.
$$
The collection $\{C_n^{\max} \}$ forms a strict, compatible matrix ordering on $\cl S\otimes \cl T$
for which $1\otimes 1$ is an Archimedean matrix order unit. We let $\cl S\otimes_{\max} \cl T$
denote the resulting tensor product. max is functorial and projective \cite{KPTT Tensor}, \cite{Han}.
By the representation of the maximal tensor product we mean the canonical identification
CP$  (\cl S \otimes_{\max} \cl T, \bb C ) \cong$CP$(\cl S, \cl T^*)$.

\subsection{Commuting tensor product} We construct the commuting tensor product $\cl S \otimes_{\rm c} \cl T$ of two operator systems $\cl S$ and $\cl T$ via ucp maps with commuting ranges, that is, the collection of matricial positive cones are defined by
$$
C^{\rm com}_n = \{ X \in M_n(\cl S \otimes \cl T): \mbox{ for all Hilbert spaces } \cl H \mbox{ and for all ucp maps }
$$
$$
\hspace{3cm} \varphi: \cl S \rightarrow B(\cl H), \; \psi: \cl T \rightarrow B(\cl H)  \mbox{ with commuting ranges we have }
$$
$$
(\varphi \otimes \psi)^n (X) \geq 0 \}.
$$
We refer \cite{KPTT Tensor} for basic properties of this tensor product and recall that if one of the tensorant has a structure of a C*-algebra then c and max coincide, that is, for an operator system $\cl S$ and a C*-algebra $\cl A$ we have a canonical complete order isomorphism $\cl S \otimes_{\rm c} \cl A = \cl S \otimes_{\max} \cl A$.

\subsection{Some asymmetric tensor products} For operator systems $\cl S$ and $\cl T$ we define the left injective and right injective tensor products by the inclusions
$$
\cl S \otimes_{\rm el} \cl T : \subseteq I(\cl S) \otimes_{\max} \cl T \;\;\;\; \mbox{and} \;\;\;\; \cl S \otimes_{\rm er} \cl T : \subseteq \cl S \otimes_{\max} I(\cl T). 
$$
The tensor product el is a left injective in the sense that for any operator systems $\cl S_1 \subseteq \cl S_2$ and $\cl T$, we have a canonical complete order embedding
$
\cl S_1 \otimes_{\rm el} \cl T \subseteq \cl S_2 \otimes_{\rm el} \cl T.
$
Besides, el is maximal left injective tensor product. Analogous properties for for the right injective tensor products hold. We refer \cite{KPTT Tensor} for details.

\subsection{Nuclearity} For operator system tensor products $\alpha$ and $\beta$ we write $\alpha \leq \beta$ if for every operator systems $\cl S$ and $\cl T$ the canonical map
$
\cl S \otimes_{\beta} \cl T \rightarrow \cl S \otimes_{\alpha} \cl T
$
is ucp. For example, the tensor products we discuss above exhibit min $\leq $ el, er $\leq$ c $\leq$ max. An operator system $\cl S$ is said to be $(\alpha,\beta)$-\textit{nuclear} if $\cl S \otimes_{\alpha} \cl T = \cl S \otimes_{\beta} \cl T $ for every operator system $\cl T$. Local liftability, weak expectaion, completely positive factorization, exactness are examples of intrinsic characterizations of nuclearity related properties. We refer the reader \cite{KPTT Nuclearity} for details and remark that, in this article, we work on (c,max)-nuclearity and (er,max)-nuclearity. An operator system $\cl S$ is said to be \textit{nuclear} if it is (min,max)-nuclear, i.e., $\cl S \otimes_{\min} \cl T = \cl S \otimes_{\max} \cl T $ for all $\cl T$.
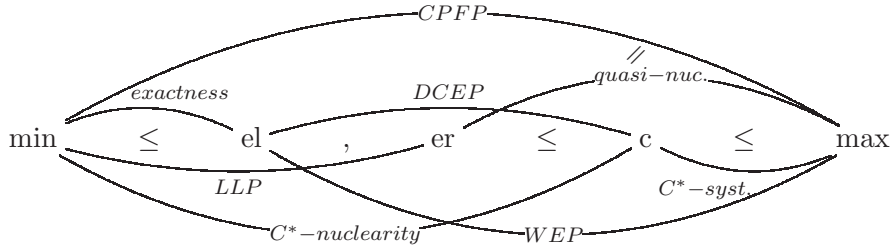
$$
\xymatrix{
\min  \ar@{-}@/^1pc/[rr]^{\hspace{1cm}exactness}   \ar@{-}@/^4pc/[rrrrrrrr]|-{CPFP} 
\ar@{-}@/_1pc/[rrrr]_{LLP}    \ar@{-}@/_3pc/[rrrrrr]|-{C^*-nuclearity}
&  \leq  &
{\rm el}  \ar@{-}@/_3pc/[rrrrrr]|-{WEP}     \ar@{-}@/^1pc/[rrrr]^{DCEP} 
&  ,   & {\rm er} \ar@{-}@/^2pc/[rrrr]^{\rotatebox{45}{=}\;\;\;\;\;}|-{quasi-nuc.}
&  \leq     &   {\rm c} \ar@{-}@/_1pc/[rr]_{\small C^*-syst.\hspace{1.3cm}} & \leq & \max
}
$$

\section{Relative Weak Injectivity}

As we defined in the introduction, an operator subsystem $\cl S_1$ of an operator system $\cl S_2$ is said to be \textit{w.r.i.} in $\cl S_2$ if the canonical inclusion of $\cl S_1$ into its bidual operator system $\cl S_1^{**}$ extends to ucp map on $\cl S_2$.
$$
\xymatrix{
 \cl S_1 \ar@{_{(}->}[d]^{j} \ar@{^{(}->}[rr]^{i} & &  \cl S_1^{**} \\
 \cl S_2  \ar[rru]_{ucp \;\tilde i} & &
}
$$

An operator system $\cl S$ is said to have the \textit{weak expectation property} (WEP) if the canonical inclusion of $\cl S \hookrightarrow \cl S^{**}$ extends to a ucp map on $I(\cl S)$. Note that this is equivalent to $\cl S$ having w.r.i$.$ in $I(\cl S)$. We start with the following:

\begin{theorem} \label{main thm}The following are equivalent for $\cl S_1 \subseteq \cl S_2:$
\begin{enumerate}
\item $\cl S_1$ is w.r.i. in $\cl S_2;$

\smallskip

\item for all matrix systems $\cl R$, every ucp map $\varphi: \cl S_1 \rightarrow \cl R$ has a ucp extension $\tilde \varphi: \cl S_2 \rightarrow \cl R;$

\smallskip

\item for every operator systems $\cl T$, every ucp map $\varphi: \cl S_1 \rightarrow \cl T^{**}$ extends to a ucp map $\tilde \varphi: \cl S_2 \rightarrow \cl T^{**};$

\smallskip

\item for every $n$ and null-subspace $J \subset M_n$ we have unital order embedding
$$
\cl S_1 \otimes_{\max} (M_n/J) \subseteq \cl S_2 \otimes_{\max} (M_n/J);
$$

\item for every operator system $\cl T$ we have a complete order embedding
$$
\cl S_1 \otimes_{\max} \cl T \subseteq \cl S_2 \otimes_{\max} \cl T;
$$

\item every state $\varphi$ on $\cl S_1$ has a state extension $\tilde \varphi$ on $\cl S_2$ such that if $\psi$ is positive linear functional on $\cl S_1$ with $\psi \leq \varphi$, then $\psi$ has positive extension $\tilde \psi$ on $\cl S_2$ with $\tilde \psi \leq \tilde \varphi$. Moreover, this can be achieved such a way that $\psi \mapsto \tilde \psi$ is a cp map from $[\varphi]$ to $[\tilde \varphi];$

\smallskip

\item $\cl S_1^{**} $ is w.r.i.\! in $ \cl S_2^{**}$, moreover the  inclusion of $\cl S_1^{**} $ into $ \cl S_2^{**}$ has a ucp inverse.

\end{enumerate}
\end{theorem}

\begin{proof}
We shall follow the pattern (1) $\Rightarrow $ (5) $\Rightarrow $ (6) $\Rightarrow $ (4) $\Rightarrow $ (2) $\Rightarrow $ (3) $\Rightarrow $ (1) $\Leftrightarrow $ (7).

\smallskip

(1) $\Rightarrow $ (5): By Lemma 6.5 of \cite{KPTT Nuclearity}, for every operator system $\cl S$ and $\cl T$ we have a complete order inclusion $\cl S\otimes_{\max} \cl T \subset \cl S^{**}\otimes_{\max} \cl T $.  Now let $j$ be the inclusion of $\cl S_1$ in $\cl S_2$ and let $\tilde i: \cl S_2 \rightarrow \cl S_1^{**}$ be the ucp map extending the canonical inclusion $ i : \cl S_1 \hookrightarrow \cl S_1^{**}$. Functoriality of the maximal tensor product ensures that
$$
\cl S_1 \otimes_{\max} \cl T \xrightarrow{j \otimes id}
\cl S_2 \otimes_{\max} \cl T \xrightarrow{\tilde i \otimes id}
\cl S_1^{**} \otimes_{\max} \cl T
$$
are ucp maps. Since the composition is a complete order embedding the first map, $j \otimes id$, must have the same property.

\smallskip

(5) $\Rightarrow $ (6): Let $\varphi$ be state on $\cl S_1$ and let $Q: \cl S_2^* \rightarrow \cl S_1^*$ be the canonical quotient map. By the assumption we have
$
\cl S_1 \otimes_{\max} \cl [\varphi] \subseteq
\cl S_2 \otimes_{\max} \cl [\varphi]. 
$
So we must have  that every positive linear functional on $\cl S_1 \otimes_{\max} \cl [\varphi]$ extends to positive linear functional on $ \cl S_2 \otimes_{\max} \cl [\varphi] $. By employing the representation of the maximal tensor product, the later condition can be rephrased as follows:  every cp map $\gamma: [\varphi] \rightarrow \cl S_1^*$ lifts to a cp map $\tilde \gamma: [\varphi] \rightarrow \cl S_2^*$ so that $Q \circ \tilde \gamma = \gamma$. In particular if we let $\gamma$ be the canonical inclusion $[\varphi] \subset \cl S_1^*$ the result follows. In fact, for each positive linear functional $\psi \leq \varphi$, $\tilde \psi = \tilde \gamma (\psi)$ yields the desired map in (6).

\smallskip

(6) $\Rightarrow $ (4): Let $n$ be a positive integer and $J$ be a null subspace in $M_n$. As a first step we will prove that $\cl S_1 \otimes_{\max} (M_n/J) \subset \cl S_2 \otimes_{\max} (M_n/J)$ order isomorphically. To do so, we need to prove that every positive linear functional $f$ on $\cl S_1 \otimes_{\max} (M_n/J)$  extends to a positive linear functional on $\cl S_2 \otimes_{\max} (M_n/J)$. By the representation of the maximal tensor product, let $\gamma_f : M_n/J \rightarrow \cl S_1^*$ be the cp map associated with $f$. By rescaling $f$, if necessary, we may suppose that $\gamma_f(I_n+J)$ is a state in $\cl S_1^*$, say $\varphi$. Let $\tilde \varphi$
be a state extension of $\varphi$ on $\cl S_2$ as in (6) and $E: [\varphi] \rightarrow [\tilde \varphi]$ be the cp map.  Define $\gamma: M_n /J \rightarrow \cl S_2^*$ by $\gamma = E \circ \gamma_f$. Let $g$ be the corresponding positive linear functional on $\cl S_2 \otimes_{\max} (M_n/J)$. It is not difficult to show that $g$ extends $f$. This proves the desired order inclusion. Now for any operator system $\cl S$ we have 
$$M_k(\cl S \otimes_{\max} (M_n/J)  ) \cong 
\cl S \otimes_{\max} M_k(M_n/J) \cong
\cl S \otimes_{\max}  (M_k\otimes M_n) / (M_k \otimes J) . 
$$
Since $M_k \otimes J$ is a null-subspace of $M_k\otimes M_n$, it follows that $\cl S_1 \otimes_{\max} (M_n/J) \subset \cl S_2 \otimes_{\max} (M_n/J)$ holds completely order isomorphically. This finishes the proof.

\smallskip

(4) $\Rightarrow $ (2): Let $\cl R \subset M_n$ be a matrix system and let $\varphi: \cl S_1 \rightarrow \cl R$ be a ucp map. Note that $\cl R^* \cong M_n^* / J$ for some null-subspace $J \subset M_n^*$. By using the completely positive identification $M_n \cong M_n^*$ via $e_{ij} \mapsto \delta_{ij}/n$, where $\{e_{ij}\}$ is the canonical matrix units in $M_n$ and $\{\delta_{ij}\}$ is the corresponding dual basis, we may suppose that $\cl R^*$
is a matrix quotient by a null-subspace. Let $f_{\varphi}: \cl S_1 \otimes_{\max} \cl R^* \rightarrow \bb C$ be the associated positive linear functional. Since  $\cl S_1 \otimes_{\max} \cl R^* \subset \cl S_2 \otimes_{\max} \cl R^*$ order isomorphically, $f_{\varphi}$ extends to a positive linear functional $\tilde f_{\varphi} : \cl S_1 \otimes_{\max} \cl R^* \rightarrow \bb C$. Now it is not difficult to show that the corresponding cp map $\psi: \cl S_2 \rightarrow \cl R$ extends $\varphi$.

\smallskip

(2) $\Rightarrow $ (3): Let $\cl T$ be an operator system and $\varphi: \cl S_1 \rightarrow \cl T^{**}$ be a ucp map. We can concretely represent $\cl T^{**} \subseteq B(\cl H)$ such a way that $\cl T^{**} \cong \overline{\cl T}^{wot}$ weak*-wot homeomorphically.  Let $\{\cl H_\alpha\}$
be the net of finite dimensional Hilbert subspaces of $\cl H$ directed under inclusion. We let $P_{\alpha}$ be the projection onto $\cl H_\alpha$. Consider the ``matrix system" given by $\cl R_{\alpha}  = P_{\alpha} \cl T^{**} P_{\alpha} \subset B(\cl H)$. We let $\varphi_\alpha : \cl S_1 \rightarrow \cl R_\alpha $ by $s \mapsto P_{\alpha} \varphi(s) P_{\alpha}$. By our assumption, $\varphi_{\alpha}$ extends to ucp map $\tilde \varphi_{\alpha}: \cl S_2 \rightarrow \cl R_\alpha$. (Note that $\varphi$ and $\tilde \varphi$ are contractive cp maps when the image is extended to $B(\cl H)$.) Let $\tilde \varphi$ be a point-weak cluster point of the net $\{\tilde \varphi_{\alpha}\}$. We let $\{\tilde \varphi_\beta\}$ be the subnet converging $\tilde \varphi$ (in point-ultraweak topology).  First note that $\tilde \varphi(s) = \varphi(s)$ for any $s \in \cl S_1$. In fact, the sequence $\{\tilde \varphi_{\beta}(s)\} = \{P_{\beta} \varphi(s) P_{\beta} \}$ converges $\varphi(s)$ in wot and hence in ultraweak topology as it's bounded. Secondly we shall prove that $\tilde \varphi(\cl S_2) \subset \overline{\cl T}^{wot}$. Let us fix $x \in \cl S_2$ and set $\tilde \varphi(x)=y$. By definition $\tilde \varphi_{\beta}(x) \in P_{\beta}\overline{\cl T} P_{\beta}$, so let $t_\beta \in \overline{\cl T}$ such that $\tilde \varphi_{\beta}(x)= P_{\beta} t_{\beta} P_{\beta}$. Note that the sequence $\{\tilde \varphi_{\beta}(x)\}$ also converges to $y$ in wot. We claim that $t_\beta$ converges $y$ in wot. Given $h_1, h_2 \in \cl H$ we fix $\beta_0$ such that $h_1,  h_2 \in \cl H_{\beta_0}$, implying that for any $\beta \geq \beta_0$ $P_{\beta} h_i = h_i$ for $i=1,2$. Now for any such $\beta$
$$
\langle t_\beta h_1,h_2\rangle  - \langle y h_1,h_2\rangle
=
\langle P_\beta t_\beta P_\beta h_1,h_2\rangle  - \langle y h_1,h_2\rangle
 = \langle (\tilde \varphi_{\beta}(x) - y) h_1,h_2\rangle \rightarrow 0.
$$
Since $\{t_\beta\} \subset \overline{\cl T}^{wot}$, $y$ must belong to $\overline{\cl T}^{wot}$. This finishes the proof.

\smallskip

(3) $\Rightarrow $ (1): Follows from definition.

\smallskip

(1) $\Rightarrow $ (7): We simply take the second adjoint of the maps appear in the definition of w.r.i.:
$$
\xymatrix{
 \cl S_1 \ar@{_{(}->}[d]^{j} \ar@{^{(}->}[rr]^{i} & &  \cl S_1^{**} \\
 \cl S_2  \ar[rru]_{\tilde i} & &
}
\;\;\;
\Rightarrow \;\;\;
\xymatrix{
 \cl S_1^{**} \ar@{_{(}->}[d]^{j^{**}} \ar@{^{(}->}[rr]^{i^{**}} & &  \cl (S_1^{**})^{**} \ar@{->>}[r]^{P}   & \cl S_1^{**}  \\
 \cl S_2^{**}  \ar[rru]_{\tilde i^{**}} & &
}
$$
The canonical inclusion of $\cl S_1^{**}$ in $(\cl S_1^{**})^{**}$ coincides with the second adjoint $i^{**}$, so $\cl S_1^{**}$ is w.r.i$.$ in $\cl S_2^{**}$.
Consider the canonical inclusion  $\cl S_1^* \hookrightarrow \cl (S_1^*)^{**}$. Let $P$ be the adjoint of this inclusion, $\cl S_1^{****} \rightarrow \cl S_1^{**}$. The composition $P\circ \tilde i^{**} $ gives the desired map from $\cl S_2^{**}$ to $\cl S_1^{**}$.

\smallskip

(7) $\Rightarrow$ (1): Let $E: \cl S_2^{**} \rightarrow \cl S_1^{**}$ be the ucp inverse of $j^{**}: \cl S_1^{**} \hookrightarrow \cl S_2^{**}$. Clearly $E|_{j(\cl S_1)}$
yields the desired extension in the definition of w.r.i.
\end{proof}

 The following is a restatement of the above theorem with the pair $\cl S \subseteq \cl I(\cl S)$. The implication (1) $\Rightarrow $ (3) was already shown in \cite{KPTT Nuclearity} and (3) $\Rightarrow$ (1) is shown in \cite{Han}.
\begin{theorem}\label{thm WEP main} The following properties of an operator system $\cl S$ are equivalent:
\begin{enumerate}

\item $\cl S$ has the weak expectation property$;$

\item for every $n$ and null-subspace $J \subset M_n$ we have a complete order isomorphism
$$
\cl S\otimes_{\min} ( M_n / J) = \cl S\otimes_{\max} ( M_n / J);
$$

\smallskip

\item for every operator system $\cl T$ we have a complete order inclusion
$$
\cl S \otimes_{\max} \cl T \subseteq \cl I(\cl S) \otimes_{\max} \cl T,
$$
in other words, $\cl S$ is (el,max)-nuclear$;$
\item $\cl S$ is approximately injective for the matrix systems in the sense that for every $n$, matrix system $\cl R \subseteq M_n$, cp map $\varphi: \cl R \rightarrow  \cl S$ and $\epsilon >0$ there is a cp map $\tilde \varphi: M_n \rightarrow \cl S$
such that $\|\tilde \varphi|_{\cl R} - \varphi \|_{cb} \leq \epsilon.$
\end{enumerate}
\end{theorem}

\begin{proof} The equivalence of (1) and (3) follows from the previous theorem. Before we proceed we remark that (2) can be rephrased as follows:

\begin{enumerate}
\item[(2')] for any $n$ and null-subspace $J \subset M_n$ we have $\cl S\otimes_{\max} ( M_n / J) \subseteq I(\cl S)\otimes_{\max} ( M_n / J)$.
\end{enumerate}
In fact, the quotient $M_n/J$ has the lifting property so the min and the max tensor products coincide on $I(\cl S) \otimes (M_n / J)$. Simply replacing max by min, and using the injectivity of min, we get the equivalence of (2) and (2'). 

\smallskip

Clearly (3) implies (2'). Now we will show that (2') implies (1). Let $\cl R \subseteq M_n$ be a matrix system and let $\varphi: \cl S \rightarrow \cl R$ be a ucp map. By the representation of the maximal tensor product let $f_\varphi: \cl S \otimes_{\max} \cl R^* \rightarrow \bb C$ be the associated positive linear functional. Since $\cl R^* \cong M_n/J$ (see Subsection 1.2) $f_\varphi$ extends to a positive linear functional $\tilde f_\varphi: I(\cl S) \otimes_{\max} \cl R^* \rightarrow \bb C$. Now using the representation of max again, we obtain a cp map $\tilde \varphi : I(\cl S) \rightarrow \cl R$, which extends $\varphi$.  By using the previous theorem we deduce that $\cl S$ is w.r.i$.$ in $I(\cl S)$, so (1) follows.

The statement in (4) can be rewritten in the tensor form:

(4') $\cl S \otimes M_n \rightarrow \cl S \otimes_{\min} (M_n/J) $ is a quotient map for any $n$ and null-subspace $J\subset M_n$.

\noindent To see this let us assume (4'). Let $\cl R \subseteq M_n$ be an operator system and let $\varphi: \cl R \rightarrow \cl S$ be a cp map. By the representation of the minimal tensor product, let $u$ be the associated positive element in $\cl S \otimes_{\min} \cl R^*.$ We declare a faithful state $f\in \cl R^*$ as unit. Now, by the assumption, $\cl S\otimes M_n^* \xrightarrow{id \otimes Q} \cl S \otimes_{\min} \cl R^*$ is a quotient map, so for any $\epsilon >0$, there exists an element $U_\epsilon \geq 0$ in $\cl S\otimes M_n^*$ such that $id_{\cl S} \otimes Q (U_\epsilon) = u + \epsilon (1_{\cl S} \otimes f)$. $U_\epsilon$ corresponds to cp map $\tilde \varphi_\epsilon: M_n \rightarrow \cl S$. One can verify that $\tilde \varphi_\epsilon|_{\cl R} - \varphi = \epsilon f(\cdot)1_{\cl S}$. Since cb-norm of a state is 1 we obtain (4). Simply reversing the steps, the reader can verify that (4) implies (4'), hence they are equivalent.

We complete our proof by showing (2) and (4') are equivalent. The projectivity of the maximal tensor product \cite{Han} ensures that the operator system structure on $\cl S \otimes ( M_n/J)$ arising from the quotient map $\cl S \otimes M_n$ coincides with the maximal tensor product. Therefore (2) and (4') are equivalent. This finishes our proof.
\end{proof}

\noindent \textbf{Question:} Let $J = span\{(1,1,1,-1,-1,-1)\} \subseteq \ell_6^\infty$. The quotient operator system $\ell^\infty_6 / J$ coincides with the dual operator system $\cl W_6^*$ \cite{FKPT-discrete}. We show in \cite{kav wri C} that, for C*-algebras $\cl A \subseteq \cl B$, $\cl A$ is w.r.i$.$ in $\cl B$ if and only if
$$
\cl A \otimes_{\max}\left( \ell_6^\infty/J \right) \subseteq \cl B \otimes_{\max} \left( \ell_6^\infty/J \right).
$$
We don't know if such a tensorial inclusion implies w.r.i$.$ for operator systems. In particular, a unital C*-algebra $\cl A$ has WEP if and only if $\cl A \otimes_{\min} \left( \ell_6^\infty/J \right) =  \cl A \otimes_{\max} \left( \ell_6^\infty/J \right) $ \cite{kavruk2012}. We don't know if a similar property characterizes WEP for general operator system.

\smallskip

The following is a restatement of Theorem \ref{main thm} for the pair $\cl S \subseteq C^*_u(\cl S)$. Technically they are all equivalent to $\cl S$ being w.r.i$.$ in $C^*_u(\cl S)$. The equivalence of (1) and (3) was known to K$.$ H$.$ Han.

\begin{theorem}\label{thm C*-main} The following are equivalent for an operator system $\cl S$:

\begin{enumerate}

\item $\cl S$ is $($c,max$)$ nuclear, that is, for any operator system $\cl T$ we have $\cl S \otimes_{\rm c} \cl T = \cl S \otimes_{\max} \cl T$;

\smallskip

\item for every $n$ and null-subspace $J \subset M_n$, we have $\cl S \otimes_{\rm c} (M_n/J) = \cl S \otimes_{\max} (M_n/J)$;

\smallskip

\item $\cl S$ is a C*-system, that is, $\cl S^{**}$ has structure of a C*-algebra.

\end{enumerate}

\end{theorem}

\begin{proof} The operator system structure on $\cl S \otimes \cl T$ arising from the inclusion $C^*_u(\cl S) \otimes_{\max} \cl T$ coincides with the commuting tensor product. Therefore the statements (1) and (2) are both equivalent to $\cl S$ being w.r.i$.$ in $C^*_u(\cl S)$. By using the universal property of $C^*_u(\cl S)$, one can easily verify that every C*-system $\cl S$ is w.r.i$.$ in $C^*_u(\cl S)$. Therefore (3) implies (1) and (2). Finally, if $\cl S$ is w.r.i$.$ in $C^*_u(\cl S)$ then $\cl S^{**}$ is w.r.i$.$  in $C^*_u(\cl S)^{**}$, with a ucp inverse. Now by a fundamental result of Choi and Effros (see Theorem 15.2 in \cite{PaulsenBook}) the ucp idempotent from $C^*_u(\cl S)^{**}$ onto $\cl S^{**}$ declares a C*-algebra structure on $\cl S^{**}$, hence $\cl S$ is a C*-system.
\end{proof}

We end this section with a brief summary of relative double commutant injectivity in the operator system category. This notion is introduced in \cite{AB} (where the author prefers to use w.r.i$.$ rather than r.d.c.i$.$). For the completeness of this work we overview the nuclearity related aspects of this property. As we defined at the introduction, an operator subsystem $\cl S_1$ of an operator system $\cl S_2$ is said to have \textit{r.d.c.i}$.$ in $\cl S_2$ if every representation $i: \cl S_1 \hookrightarrow B(\cl H)$ extends to a ucp map $\tilde i: \cl S_2 \rightarrow B(\cl H)$ such that $\tilde i (\cl S_2) \subseteq i(\cl S_1)''$. We remark that an operator system $\cl S$ has r.d.c.i$.$ in its injective envelope $I(\cl S)$ if and only if it has the double commutant injectivity property (see Section 7 of \cite{KPTT Nuclearity} for this nuclearity related property.) In the following $C^*(\bb F_\infty)$ denotes the full group C*-algebra of the free group on countably infinite number of generators.  

\begin{theorem}[\cite{AB}] The following are equivalent for $\cl S_1 \subseteq \cl S_2:$

\begin{enumerate}

\item $\cl S_1$ has r.d.c.i$.$ in $\cl S_2$;

\item for any operator system $\cl T$ we have $\cl S_1 \otimes_{\rm c} \cl T \subseteq \cl S_2 \otimes_{\rm c} \cl T$;

\item for any unital C*-algebra $\cl A$ we have $\cl S_1 \otimes_{\max} \cl A \subseteq \cl S_2 \otimes_{\max} \cl A$;

\item we have $\cl S_1 \otimes_{\max} C^*(\bb F_\infty) \subseteq \cl S_2 \otimes_{\max} C^*(\bb F_\infty)$;

\item $C_u^*(\cl S_1)$ is w.r.i$.$ in $C_u^*(\cl S_2)$.

\end{enumerate}

\end{theorem}

\section{C*-systems and Quasi-nuclearity}

An operator system $\cl S$ is said to be \textit{nuclear} if it is (min,max)-nuclear, that is, for every operator system $\cl T$ the minimal and the maximal tensor product on $\cl S \otimes \cl T$ coincide. A unital C*-algebra $\cl A$ is classically defined as \textit{nuclear} if $\cl A \otimes_{\min} \cl B = \cl A \otimes_{\max} \cl B $ for every C*-algebra $\cl B$. It is elementary to verify that $\cl A$ is a nuclear C*-algebra if and only if $\cl A $ is nuclear as an operator system. In a recent study \cite{Kavruk NP} we extend a classical result of Namioka and Phelps \cite{NP} to non-commutative setting by proving that Namioka and Phelps' test system 
$$\cl W_6  = 
\{ (a_i)_{i=1}^6 : a_1 + a_2 + a_3 =  a_4 + a_5 + a_6   \} \subseteq \ell_6^\infty$$ detects nuclear C*-algebras. More precisely, a unital C*-algebra $\cl A$ is nuclear if and only if the minimal and the maximal tensor product on $\cl A \otimes \cl W_6$ coincide. Unfortunately such a property for general operator systems remains open. Here is an extension to C*-systems:

\begin{theorem} A C*-system  $\cl S$ is nuclear if and only if we have a canonical complete order isomorphism
$\cl S \otimes_{\min} \cl W_{6} = \cl S \otimes_{\max} \cl W_{6}$.
\end{theorem}
\begin{proof} We will only prove the non-trivial direction $(\Leftarrow)$. Let $i$ be a von Neumann algebra  embedding of $\cl S^{**} $ into a  $B(\cl H)$, and let $i_0 = i|_{\cl S}$. Let $[i_0]\subseteq CB(\cl S, B(\cl H))$ be the Effros system. We will first prove that $[i_0]\cong i_0(S)' \subseteq B(\cl H)$ then show that $[i_0]$ has WEP. First observe that we have a canonical embedding of $i_0(S)'$ into $[i_0]$ given by $A\mapsto Ai_0$ (here $Ai_0$ is given by $s\mapsto Ai_0(s)$). We also have a natural surjective map $R: [i] \rightarrow [i_0]$ given by the restriction: for a cp map $\psi: S^{**} \rightarrow B(\cl H)$ with $\psi \leq i$ we define $R(\psi) = \psi|_S$. Surjectivity follows from the fact that whenever $\psi \leq i_0$ then the weak extension $\psi^{**}: \cl S^{**} \rightarrow B(\cl H)$ satisfies $\psi^{**} \leq i$ and  $R(\psi^{**}) = \psi$. This shows that $R$ is a quotient map. In this particular case, by Radon-Nikodym theory of Arveson \cite{Arveson},  $[i]$ coincides with $i(\cl S^{**})' = i_0(\cl S)'$. As $R$ is both injective and projective it must be a bijective complete order isomorphism. This proves that $[i_0]\cong i_0(S)'$, in particular, $[i_0]$ has structure of a von Neumann algebra. Next we will show that $[i_0]$ has WEP, equivalently is injective.  Our assumption is equivalent to the statement that
$$
\cl S \otimes_{\max} \cl W_{6} \subseteq \cl S \otimes_{} \ell^\infty_6
$$
completely order isomorphically.
Let $\gamma: \cl W_6 \rightarrow [i_0]$ be a cp map. By the representation of the maximal tensor product we obtain a cp map $\Gamma:\cl S \otimes_{\max} \cl W_6 \rightarrow B(\cl H)$. $\Gamma$ extends to a cp map $\tilde \Gamma$ from $\cl S \otimes_{} \ell^\infty_6$ into  $ B(\cl H)$. Going backward we obtain a cp map $\tilde \gamma: \ell_6^\infty \rightarrow [i_0]$.
By \cite{kavruk2012} it follows that $[i_0]\cong i_0(S)'$ has WEP. A fundamental result of Effros and Lance \cite[Prop.3.7]{EffrosLance}, then, implies that $i_0(\cl S)''= i(\cl S^{**})'' = i(\cl S^{**}) \cong \cl S^{**}$ is an injective von Neumann algebra. Finally, by Kirchberg's characterization of nuclearity \cite{Kirchberg94}, we conclude that $\cl S$ is nuclear.
\end{proof}

\noindent \textbf{Question:} If $\cl S \otimes_{\min} \cl W_{6} = \cl S \otimes_{\max} \cl W_{6}$ can we conclude that $\cl S$ is nuclear?

\smallskip

\noindent \textit{Remark:} In the previous theorem we can replace $\cl W_6$ by the 4-dimensional operator system
$$
\cl W_{2,3}=\{ (a_1, ..., a_5): 3(a_1 + a_2) = 2(a_3+a_4+a_5)  \} \subseteq \ell_5^\infty.
$$
In fact $\cl W_{2,3}^*$ detects WEP for C*-algebras \cite{kavruk2012}.

\smallskip

A C*-algebra $\cl A$ is nuclear if and only if $\cl A \otimes_{\max} \cl B_1 \subseteq  \cl A \otimes_{\max} \cl B_2 $ for every C*-algebras $\cl B_1 \subseteq \cl B_2$, where that later condition is called \textit{quasi-nuclearity} (or \textit{semi-nuclearity}) \cite{Lance}. In fact if $\cl A$ is nuclear we can replace max by min, so the result simply follows from the injectivity of the minimal tensor product. Conversely, by fixing a representation $\cl A \subseteq B(\cl H)$ with $\cl A^{**} \cong \cl A^{''}$, a canonical C*-algebra inclusion $\cl A \otimes_{\max} \cl A' \subseteq \cl A \otimes_{\max} B(\cl H)$, if holds, is solely sufficient to conclude that $\cl A$ is nuclear. In fact the canonical ucp map $\cl A \otimes_{\max} \cl A' \ni a\otimes a' \mapsto aa' \in B(\cl H)$ extends to a ucp map on $\cl A \otimes_{\max} B(\cl H)$, say $\varphi$. Employing Choi's theory of multiplicative domain, $\varphi|_{B(\cl H)}$ must have an image sitting in $\cl A'$. We conclude that $\cl A'$ is injective, therefore $\cl A''\cong \cl A^{**}$ is injective, which is sufficient to conclude that $\cl A$ is nuclear \cite{Effros}. The following result is an operator system analogue of quasi-nuclearity.

\begin{theorem}
The following are equivalent for an operator system $\cl S$:
\begin{enumerate} 

\item $\cl S$ is nuclear;

\item $\cl S$ is (er,max)-nuclear;

\item $\cl S$ is quasi-nuclear, that is, for every operator systems $\cl T_1 \subseteq \cl T_2$ we have
$$
\cl S \otimes_{\max} \cl T_1 \subseteq  \cl S \otimes_{\max} \cl T_2;
$$

\item for any matrix system $\cl R$ we have
$
\cl S \otimes_{\min} \cl R = \cl S \otimes_{\max} \cl R.
$

\end{enumerate}
\end{theorem}

\begin{proof} We first observe that (2) and (3) are equivalent. In fact, given $\cl T_1 \subseteq \cl T_2$, the right injectivity of er guarantees that $\cl S \otimes_{\rm er} \cl T_1 \subseteq  \cl S \otimes_{\rm er} \cl T_2.$ If $\cl S$ is (er,max)-nuclear then we can replace er by max, so (3) follows. Conversely, assuming (3), given an operator system $\cl T$, we have the inclusion $\cl S \otimes_{\max} \cl T \subseteq \cl S \otimes_{\max} I(\cl T)$. As the operator system structure on $\cl S \otimes \cl T$ arising from the inclusion $\cl S \otimes_{\max} I(\cl T)$ is er, we conclude that $\cl S$ is (er,max)-nuclear.

\smallskip

Clearly (1) implies equivalent conditions (2) and (3). Here we will prove that (3) implies (1): as a first step we show that $\cl S$ is a C*-system. In fact, for an operator system $\cl T$, our assumption guarantees that we have an inclusion $\cl S \otimes_{\max} \cl T \subseteq \cl S \otimes_{\max} C^*_u(\cl T) $. As the operator system structure on $\cl S\otimes \cl T$ arisen from the inclusion $\cl S \otimes_{\max} C^*_u(\cl T)$ coincides with the commuting tensor product, we obtain that $\cl S \otimes_{\rm c} \cl T  = \cl S \otimes_{\max} \cl T$. Thus, $\cl S$ is (c,max)-nuclear, equivalently, by Theorem \ref{thm C*-main}, $\cl S$ is a C*-system. Now, our assumption also yields that we have a complete order inclusion $\cl S \otimes_{\max} \cl W_6 \subseteq \cl S \otimes_{\max} \ell_6^\infty $. As $\ell_6^\infty$ is nuclear, the later inclusion yields that $\cl S \otimes_{\min} \cl W_{6} = \cl S \otimes_{\max} \cl W_{6}$. Hence, by the above theorem, $\cl S$ is nuclear.

\smallskip

(1) clearly implies (4). Therefore it suffices to prove that (4) implies (3). Given a matrix system $\cl R \subseteq M_n$, the condition in (4) is equivalent to the canonical complete order embedding $\cl S \otimes_{\max} \cl R \subseteq \cl S \otimes M_n$. Let $f$ be a state on $\cl S$ and let $[f] \subseteq \cl S^*$ be the corresponding Effros system. As a first step we wish to prove that $[f]$ has structure of an injective von Neumann algebra. Let $\varphi: \cl R \rightarrow [f] (\subseteq \cl S^*)$ be a cp map. By the representation of the maximal tensor product we obtain a positive linear functional $\Gamma_{\varphi}: \cl S \otimes_{\max} \cl R \rightarrow \bb C$. By our assumption, $\Gamma_\varphi$ extends to a positive linear functional $\tilde \Gamma_\varphi : \cl S \otimes M_n \rightarrow \bb C $. Employing the representation of the max again, we obtain a cp map $\tilde \varphi : M_n \rightarrow \cl S^*$, which extends $\varphi$. As $\tilde \varphi(1) = \varphi(1)$ and $[f]$ includes all the positive linear functionals dominated by $f$, we conclude that image of $\tilde \varphi$ sits in $[f]$. This shows $[f]$ is approximately injective (actually injective) for the matrix systems, hence, by Theorem \ref{thm WEP main} it has WEP. This means that we must have the canonical complete order embedding $[f] \otimes_{\max} \cl S \subseteq I([f]) \otimes_{\max} \cl S$. Now consider the identity map $i:[f] \rightarrow [f] (\subseteq S^*)$, which corresponds to a positive linear functional $\gamma$ on $[f] \otimes_{\max} \cl S $. By the previous inclusion, $\gamma$ extends to a positive linear functional $\tilde \gamma$ on $I([f]) \otimes_{\max} \cl S$. By using the representation of the max again, we obtain a cp map $\tilde i : I([f]) \rightarrow \cl S^*$, which extends $i$. Clearly the image of $\tilde i$ sits in $[f]$. Therefore we obtain a ucp extension $\tilde i : I([f]) \rightarrow [f]$ of $i$. A fundamental result of Choi and Effros (\cite[Thm. 15.2]{PaulsenBook}) implies that $[f]$ must have a structure of a C*-algebra. Finally, we observe that every increasing bounded net of positive elements in $[f]$ must have a supremum. We leave the details to the reader and conclude that $[f]$ must have a structure of a von Neumann algebra. As WEP and injectivity coincides for von Neumann algebras we obtain the first part of our proof. Now, let $\cl T_1 \subseteq \cl T_2$ be given. Let $f: \cl S \otimes_{\max} \cl T_1 \rightarrow \bb C$ be a positive linear functional. It suffices to show that $f$ extends to a positive linear functional $ \tilde f :\cl S \otimes_{\max} \cl T_2 \rightarrow \bb C$. Let  $\varphi_f : \cl T_1 \rightarrow \cl S^*$ be the cp map corresponding $f$. By rescaling we may suppose that $\varphi_f (1)$ is a state. Let $g$ denotes this state. Clearly the image of $\varphi_f$ sits in $[g] \subseteq \cl S^*$. As $[g]$ is an injective object, $\varphi_f$ extends to a cp map $\tilde \varphi_f : \cl T_2 \rightarrow [g](\subseteq \cl S^*)$. Now let $\tilde f :  \cl S \otimes_{\max} \cl T_2 \rightarrow \bb C$ be the associated positive linear functional. It is elementary to verify that $\tilde f$ extends $f$. This shows that we have an order inclusion $\cl S \otimes_{\max} \cl T_1 \subseteq \cl S \otimes_{\max} \cl T_2$. Finally tensoring with $M_n$ one can show that the inclusion holds completely order isomorphically. This finishes our proof.
\end{proof}

\noindent \textbf{Question:} An operator system $\cl S$ is said to be \textit{C*-quasi-nuclear} if for every $\cl A \subseteq \cl B$, where $\cl A$ is a  C*-subalgebra of the C*-algebra $\cl B$, we have
$
\cl S \otimes_{\max} \cl A \subseteq  \cl S \otimes_{\max} \cl B.
$
Clearly, C*-nuclearity implies C*-quasi-nuclearity. We don't know if the inverse is true. C*-nuclearity is well known to be equivalent to (min,c)-nuclearity \cite{KPTT Nuclearity}. It can be verified that C*-quasi-nuclearity is equivalent to (er,c)-nuclearity.

\smallskip

We end this section with the following trio for an operator system $\cl S$ \cite{Kirchberg94}, \cite{KPTT Nuclearity}, \cite{Han}:

\smallskip

$\cl S$ is (min,max)-nuclear $\Longleftrightarrow$  $\cl S^{**}$ is an injective von Neumann algebra;

\smallskip

$\cl S$ is (el,max)-nuclear $\Longleftrightarrow$  $\cl S^{**}$ is a von Neumann algebra which is injective relative to $\cl S$ 

\hspace{10cm}(i.e$.$ $\cl S$ has WEP);

$\cl S$ is (c,max)-nuclear $\Longleftrightarrow$  $\cl S^{**}$ is a von Neumann algebra.

$ $

{\sc  Department of Mathematics \& Applied Mathematics,

Virginia Commonwealth University

Richmond, VA 23220, U.S.A.}

{\it E-mail address:} askavruk@vcu.edu


\smallskip












\begin{thebibliography}{10}



\bibitem{Arveson} W. B. Arveson, Subalgebras of C*-algebras I, Acta Math. 123 (1969) 141-224.


\bibitem{AB} A. Bhattacharya, Relative weak injectivity of operator system pairs,\textit{ J. Math. Anal. Appl.}, Vol$.$ 420, Issue 1, (2014)


\bibitem{Blecher} D$.$ P$.$ Blecher, B$.$ L$.$ Duncan, Nuclearity-related properties for nonselfadjoint algebras, \textit{J. Op. Theory}
Vol$.$ 65, Issue 1 (2011)




\bibitem{ChoiEffros77}  M. Choi, E. Effros, Injectivity and operator spaces, \textit{J. Functional Analysis} 24 (1977) 156-209.










\bibitem{Effros} E. Effros, Aspects of non-commutative order,
in C*-algebras and applications to physics (Proc. 2nd Japan-USA Seminar, Los Angeles, 1977), eds.




\bibitem{EffrosLance} E. Effros, C. Lance, Tensor products of operator algebras, {\it Adv. in Math.}, 25 (1977), pp. 1-34









\bibitem{FKPT-discrete}
D.~Farenick, A.~S. Kavruk, V.~I. Paulsen and I.~G. Todorov,
\newblock Operator systems from discrete groups,
{\it Comm. Math. Phys.}, Volume 329, Issue 1 (2014), Page 207-238)


\bibitem{farenick--paulsen2011}
D.~Farenick, V.~I. Paulsen,
\newblock Operator system quotients of matrix algebras and their tensor products, {\it Mathematica Scandinavica}, Vol. 111 (2012)





\bibitem{Han} K. H. Han, 
\newblock On maximal tensor products and quotient maps of operator systems, \newblock \textit{J. Math. Anal. Appl.} 384(2):12 (2011)


\bibitem{Haagerup} U. Haagerup, Self-polar forms, conditional expectations and the weak expectation property for C∗-algebras. Unpublished manuscript (1995).












 
\bibitem{kavruk2011}
A.~S. Kavruk,
\newblock Nuclearity related properties in operator systems.
{\it J. Op. Theory}, V 71 Issue 1 (2014) 






\bibitem{Kavruk NP} A.~S. Kavruk, On a non-commutative analogue of a classical result of Namioka and Phelps, \textit{J. Funct. Anal.}
Vol 269, Issue 10, (2015), Pages 3282-3303





\bibitem{kav wri C}
A.~S. Kavruk,
\newblock Relative weak injectivity for C*-algebras.
\newblock {\em preprint (arXiv:1610.08599 )}, 2016.

\bibitem{kavruk2012}
A.~S. Kavruk,
\newblock The weak expectation property and Riesz interpolation,
\newblock {\em preprint (arXiv:1201.5414)}, 2012.





\bibitem{KPTT Nuclearity}
A.~S. Kavruk, V.~I. Paulsen, I.~G. Todorov, and M.~Tomforde,
\newblock Quotients, exactness and nuclearity in the operator system category, {\it Adv. Math.} Volume 235, 1 March 2013, Pages 321-360

\bibitem{KPTT Tensor} A.~S. Kavruk, V.~I. Paulsen, I.~G. Todorov, and M.~Tomforde,
\newblock Tensor Products of Operator Systems,
{\it J. Funct. Anal.}, Vol 261, Issue 2, (2011)




\bibitem{Kirchberg94} E. Kirchberg, On non-semisplit extensions, tensor products and exactness of group C*-algebras, {\it Invent. Math.} 112 (1993) 449-489.



\bibitem{Kirchberg-Presentation} E. Kirchberg,
Some properties of QWEP C∗-algebras, presentation at University of Copenhagen, 2012

\bibitem{KW} E. Kirchberg, S. Wassermann,
C*-Algebras Generated by Operator Systems,
{\it J. Funct. Anal.} V 155, Issue 2, (1998)


\bibitem{Lance} C. Lance, \newblock Tensor products and nuclear C*-algebras, \newblock \textit{Proceedings of Symposia in Pure Mathematics,} Vol.
38 (1982) 379-399.


\bibitem{Lance2} C. Lance, On nuclear C*-algebras, \textit{J. Functional Analysis} 12(1973), 157-176.


\bibitem{JIAN LIANG} J$.$ Liang, Operator-Valued Kirchberg Theory and Its Connection to Tensor Norms and Correspondence, \textit{Doctoral Dissertation}, University of Illinois Urbana-Champaign (2015) 





\bibitem{NP} I$.$ Namioka and R$.$ R$.$ Phelps,
\newblock Tensor products of compact convex sets,
\newblock \textit{Pacific J. Math.} Volume 31, Number 2 (1969), 469-480





\bibitem{PaulsenBook} V. I. Paulsen, {\it Completely bounded maps and operator algebras}, Cambridge Studies in Advanced Mathematics 78, Cambridge University Press, 2002.

\bibitem{Paulsen WEP} V. I. Paulsen, Weak Expectations and the Injective Envelope, \textit{Trans. Amer. Math. Soc.} V. 363 (2011)

\bibitem{Paulsen-Tomforde} V. I. Paulsen and M. Tomforde, Vector spaces with an order unit, Indiana Univ. Math. J., 58 (3) (2007)





\bibitem{pisier_intr}
G. Pisier, {\it Introduction to Operator Space Theory}, {\rm
Cambridge University Press, 2003}





\end{thebibliography}
\end{document}